\documentclass[11pt, oneside]{amsart}

\usepackage{amsthm}
\usepackage{amsmath}
\usepackage{amsfonts}
\usepackage{amssymb}
\usepackage[mathscr]{eucal}

\newtheorem{theirtheorem}{Theorem}

\theoremstyle{plain}
\newtheorem{theorem}{\textbf{Theorem}}[section]
\newtheorem{lemma}[theorem]{\textbf{Lemma}}

\newtheorem*{Freiman-3k-4}{\textbf{Freiman $3k-4$ Theorem}}
\newtheorem*{Freiman-3k-3}{\textbf{Freiman $3k-3$ Theorem}}

\newcommand{\Z}{\mathbb{Z}}

\newcommand{\be}{\begin{equation}}
\newcommand{\ee}{\end{equation}}
\newcommand{\Summ}[1]{\underset{#1}{\sum}}

\newcommand{\C}{\mathbf{C}}
\newcommand{\R}{\mathbb{R}}

\newcommand{\ber}{\begin{eqnarray}}
\newcommand{\eer}{\end{eqnarray}}
\newcommand{\ba}{\begin{align}}
\newcommand{\ea}{\end{align}}

\newcommand{\nn}{\nonumber}
\newcommand{\x}{\textbf{x}}
\newcommand{\e}{\textbf{e}}

\newcommand{\epi}{\mathsf{epi}\,}
\newcommand{\und}{\;\mbox{ and }\;}

\begin{document}

\title{Inverse Additive Problems for Minkowski Sumsets II}
\author{G.~A.~Freiman} %\inst{1}
\address{The Raymond and Beverly Sackler Faculty of Exact Sciences School of
Mathematical Sciences, Tel Aviv University.}
\email{grisha@post.tau.ac.il}
% \thanks{Supported  .}%
\author{D.~Grynkiewicz}
\address{Institut f\"{u}r Mathematik und Wissenschaftliches Rechnen. Karl-Franzens-Universit\"{a}t, Graz.}
\email{diambri@hotmail.com}
% \thanks{Partially supported by}
\thanks{Partially supported by the Spanish Research Council project  MTM2008-06620-C03-01 and the
FWF Austrian Scient Fund Project P21576-N18}
\author{O. Serra} %\inst{2}
\address{Departament de Matem\`atica Aplicada IV,
        Universitat Polit\`ecnica de Catalunya.}
\email{oserra@ma4.upc.edu}
\thanks{Supported by the Catalan Research Council    under project .}%
\author{Y.~V.~Stanchescu}% \inst{3}
\address{The Open University of Israel, Raanana 43107 and Afeka Academic College, Tel Aviv 69107.}
\email{ionut@openu.ac.il and yonis@afeka.ac.il}
\subjclass[2010]{52A20, 52A40, 26B25}
\keywords{Brunn-Minkowski, convex bodies, sumset, convex functions}

\begin{abstract}
The Brunn-Minkowski Theorem asserts that $\mu_d(A+B)^{1/d}\geq
\mu_d(A)^{1/d}+\mu_d(B)^{1/d}$ for convex bodies $A,\,B\subseteq
\R^d$, where $\mu_d$ denotes the $d$-dimensional Lebesgue measure.
It is well-known that equality holds if and only if $A$ and $B$ are
homothetic, but few characterizations of equality in other related
bounds are known. Let $H$ be a hyperplane. Bonnesen later
strengthened this  bound by showing $$\mu_d(A+B)\geq
\left(M^{1/(d-1)}+N^{1/(d-1)}\right)^{d-1}\left(\frac{\mu_d(A)}{M}+\frac{\mu_d(B)}{N}\right),$$
where $M=\sup\{\mu_{d-1}((\mathbf x+H)\cap A)\mid \mathbf x\in
\R^d\}$ and $N=\sup\{\mu_{d-1}((\mathbf y+H)\cap B)\mid \mathbf y\in
\R^d\}$. Standard compression arguments show that the above bound
also holds when $M=\mu_{d-1}(\pi(A))$ and $N=\mu_{d-1}(\pi(B))$,
where $\pi$ denotes a projection of $\R^d$ onto $H$, which gives an
alternative generalization of the Brunn-Minkowski bound. In this
paper, we characterize the cases of equality in this later bound,
showing that equality holds if and only if $A$ and $B$ are obtained
from a pair of homothetic convex bodies by `stretching' along the
direction of the projection, which is made formal in the paper. When
$d=2$, we  characterize the case of equality in the former bound as
well.
\end{abstract}

\maketitle

\section{Introduction}
Let $\R^d$ denote the $d$-dimensional euclidian space equipped with
the usual  Lebesgue measure. Let $A,\,B\subseteq \R^d$ be {\it
convex bodies}, meaning that $A$ and $B$ are compact, convex subsets
with nonempty interior. Their Minkowski sum, or sumset, is
$$A+B=\{\mathbf a+\mathbf b\mid \mathbf a\in A,\,\mathbf b\in B\}.$$
Whenever the dimension of the convex body $A$ is clear, we will use
$|A|$ to denote its corresponding non-zero Lebesgue measure. For
$\lambda\in \R$, let $\lambda A=\{\lambda \mathbf a\mid \mathbf a\in
A\}$ denote the dilation of $A$ by $\lambda$.  The classical
Brunn-Minkowski  Theorem  gives a lower bound for $|A+B|$ in terms
of $|A|$ and $|B|$, and there are many far reaching generalizations
and applications; see \cite{Gardner-survey} for a fairly
comprehensive survey. Equality is known to hold if and only if $A$
and $B$ are {\it homothetic}, that is, $A=\lambda B+\mathbf v$ for
some $\lambda>0$ and $\mathbf v\in \R^d$ \cite{henstock-macbeath,
Gardner-survey}.

\begin{theirtheorem}[Brunn-Minkowski Theorem]\label{brunn-minkowski-thm}
If $A,\,B\subseteq \R^d$ are convex bodies, then
\be\label{brunn-minkowski-bound}
|A+B|\geq \left(|A|^{1/d}+|B|^{1/d}\right)^{d}.
\ee
\end{theirtheorem}

For $M,\,N>0$, it can be  shown (as remarked in
\cite{dancs-uhrin,Gardner-survey}) that
\ber\label{bonneson-is-better-brunn-mink}
\left(M^{1/(d-1)}+N^{1/(d-1)}\right)^{d-1}\left(\frac{|A|}{M}+\frac{|B|}{N}\right)\geq
\left(|A|^{1/d}+|B|^{1/d}\right)^{d},\eer
with equality only when
$$M |B|^{\frac{d-1}{d}}=N |A|^{\frac{d-1}{d}}.$$
Consequently, the following result given by Bonnesen in 1929 (see
e.g. \cite{bonnesen-original, bonnesen-book, dancs-uhrin,
Gardner-survey}) improves the Brunn-Minkowski Inequality. Note,
since $A$ and $B$ are compact with nonempty interiors, that the
values $M$ and $N$ in Theorem \ref{Bonnesen-thm} are nonzero and
actually attained for some $\mathbf x\in \R^d$ and $\mathbf
y\in\R^d$. For $d=1$, the coefficients of $|A|$ and $|B|$ in Bonnesen's Bound are to be interpreted as their natural limiting values, i.e., $|A+B|\geq |A|+|B|$.

\begin{theirtheorem}[Bonnesen's Bound I]\label{Bonnesen-thm}
If $A,\,B\subseteq \R^d$ are convex bodies and $H\subseteq \R^d$ is a $(d-1)$-dimensional subspace, then
\be\label{bonneson-bound-hyperslice}
|A+B|\geq \left(M^{1/(d-1)}+N^{1/(d-1)}\right)^{d-1}\left(\frac{|A|}{M}+\frac{|B|}{N}\right),
\ee
where
$M=\sup\{|(\mathbf x+H)\cap A|\mid \mathbf x\in \R^d\}$ and $N=\sup\{|(\mathbf y+H)\cap B|\mid \mathbf y\in \R^d\}$.
\end{theirtheorem}

By standard symmetrization or compression arguments (see e.g.
\cite{convexcalculus, oriol-2d-result} or the proof of Lemma
\ref{projection-is-homothetic-lemma}), Theorem \ref{Bonnesen-thm}
implies the following alternative generalization of the
Brunn-Minkowski Theorem.

\begin{theirtheorem}[Bonnesen's Bound II]\label{Bonnesen-cor}
If $A,\,B\subseteq \R^d$ are convex bodies and $\pi:\R^d\rightarrow \R^{d}$ is a linear transformation with $\dim(\ker \pi)=1$, then
\be\label{bonneson-bound-proj}
|A+B|\geq \left(M^{1/(d-1)}+N^{1/(d-1)}\right)^{d-1}\left(\frac{|A|}{M}+\frac{|B|}{N}\right),
\ee
where $M=|\pi(A)|$ and $N=|\pi(B)|$.
\end{theirtheorem}

In fact, Theorems \ref{brunn-minkowski-thm}, \ref{Bonnesen-thm} and
\ref{Bonnesen-cor} remain true for any subsets $A,\,B\subseteq \R^d$
such that all involved quantities are measurable (see
\cite{henstock-macbeath}). However, the general measurable case is
rather painful from a technical point of view, and it is a rare
textbook that is willing to reproduce the full proof of the case of
inequality in Theorem \ref{brunn-minkowski-thm} for measurable
subsets. To avoid similar issues and present our ideas with greater
clarity, we have focused here only on the case of convex bodies. The
formulation given in Theorem \ref{Bonnesen-cor} actually arises
naturally when attempting to give a discrete version of the
Brunn-Minkowski Theorem valid in $\Z^d$; see \cite{oriol-2d-result,
2-dim-paper}, or \cite{Gardner-gronchi} for a discrete version of a
somewhat different form.

\medskip

We will use the following notation throughout the paper.
Let $\pi:\R^d\rightarrow \R^d$ be a linear transformation with
$\dim(\ker \pi)=1$. Then $\pi(\R^d)=K$ for some $(d-1)$-dimensional
subspace $K$. Let $\e_0,\e_1,\ldots,\e_{d-1}\in \R^{d}$ be an
orthonormal basis for $\R^d$ such that  $\e_1,\ldots,\e_{d-1}$ span
$K$. Since $\dim(\ker \pi)=1$, we have  $\ker \pi=\R\mathbf u$ for
any nonzero $\mathbf u\in \ker \pi$.
Choose $\mathbf u\in \ker \pi$
such that the elements $\mathbf u,\e_1,\ldots,\e_{d-1}$ form a basis for $\R^d$
with the linear isomorphism $\varphi:\R^d\rightarrow \R^d$ defined
by $\varphi(\e_i)=\e_i$ for $i\geq 1$ and $\varphi(\e_0)=\mathbf u$
being volume preserving.

Then an element $\mathbf x=x_0\mathbf
u+x_1\e_1+\ldots+x_{d-1}\e_{d-1}\in \R^d$ may be written as $\mathbf
x=(x_0,x_1,\ldots,x_{d-1})$ and a convex body $A\subseteq \R^d$ can
be described as
\begin{equation}\label{convex-graph-characterization}
A=\{(y,\mathbf x) \in \R\times \R^{d-1}\mid \mathbf x\in \pi(A),\; u_A(\mathbf x)\le y \le v_A(\mathbf x)\}
\end{equation}
with  $u_A:\pi(A)\subseteq \R^{d-1}\rightarrow \R$ a convex function  and $v_A:\pi(A)\subseteq\R^{d-1}\rightarrow \R$ a concave function. We say that $A'$ is
a  {\it stretching} of $A$ (with respect to $\pi$) of amount
$h \ge 0$ if
$$
A'=\{(y,\mathbf x)\in \R\times \R^{d-1}\mid \mathbf x\in \pi(A),\; u_A(\mathbf x)\le y \le v_A(\mathbf x)+h\}.
$$
When $\mathbf u=\e_0$, which we will be able to assume as a
normalization condition as explained at the beginning of Section
\ref{main-sec}, we speak of a \emph{vertical stretching}.

\medskip

The goal of this paper is to characterize the pairs $A$ and $B$ for which equality holds in Theorem \ref{Bonnesen-cor}.

\begin{theorem}\label{Dual-thm} Let $A,\,B\subseteq \R^d$ be convex
bodies and let $\pi:\R^d\rightarrow \R^{d}$ be a linear
transformation with $\dim(\ker \pi)=1$. Then
\be\label{dual-hypothesis}|A+B|=
\left(M^{1/(d-1)}+N^{1/(d-1)}\right)^{d-1}\left(\frac{|A|}{M}+\frac{|B|}{N}\right),\ee
where $M=|\pi(A)|$ and $N=|\pi(B)|$, if and only if there are
homothetic convex bodies $A',\,B'\subseteq \R^d$ such that $A$ is a
stretching of $A'$ and $B$ is a stretching of $B'$, both with
respect to $\pi$.
\end{theorem}

When $d=2$, we also give a simple argument to derive the
characterization of equality in Theorem \ref{Bonnesen-thm} from the
characterization of equality in Theorem \ref{Bonnesen-cor}.

\begin{theorem}\label{thm-2-dim-nondual} Let $H\subseteq \R^2$ be a one dimensional subspace and let $A,\,B\subseteq \R^2$
be convex bodies translated so that
\begin{align}
\nn&M:=|H\cap A|=\sup\{|(\mathbf x+H)\cap A|\mid \mathbf x\in \R^2\}\und\\\nn &N:=|
H\cap B|=\sup\{|(\mathbf x+H)\cap B|\mid \mathbf x\in \R^2\}.
\end{align}
Then
\be\label{dim2-hyp}
|A+B|= (M+N)\left(\frac{|A|}{M}+\frac{|B|}{N}\right)
\ee
if and only if there exists a linear transformation
$\pi:\R^2\rightarrow \R^2$  and  homothetic convex bodies
$A',\,B'\subseteq \R^2$ such that $\pi(\R^2)=H$,
$$\pi(A)=\pi(A')=H\cap A=H\cap A'\und\pi(B)=\pi(B')=H\cap B=H\cap
B'$$ with $A$ a stretching of $A'$ and $B$ a stretching of $B'$,
both with respect to $\pi$.
\end{theorem}

\section{Equality in the Projection Bonnesen Bound}\label{main-sec}

The goal of this section is to prove Theorem \ref{Dual-thm}. The
case $d=1$ is trivial, so we henceforth assume $d\geq 2$. We use the
notation introduced before Theorem \ref{Dual-thm}. Then, letting
$\pi':\R^d\rightarrow \R^d$ denote the projection given by
$\pi'(y_0\e_0+y_1\e_1+\ldots+y_{d-1}\e_{d-1})=
y_1\e_1+\ldots+y_{d-1}\e_{d-1}$, we have
$|\pi(A)|=|\pi'(\varphi^{-1}(A))|$ and
$|\pi(B)|=|\pi'(\varphi^{-1}(B))|$. Since $\varphi$ is volume
preserving, and hence $\varphi^{-1}$ as well, we see that it
suffices to prove the theorem when $\mathbf u=\e_0$, as we can then
apply this case of the theorem to $\varphi^{-1}(A)+\varphi^{-1}(B)$,
derive the structure of $\varphi^{-1}(A)$ and $\varphi^{-1}(B)$, and
then find the structure of $A$ and $B$ by applying the linear
isomorphism $\varphi$. Thus we assume $\mathbf u=\e_0$ throughout
this section. In particular, $\pi:\R^d\rightarrow \R^d$ denotes the
projection given by
$$\pi(x_0,x_1,\ldots,x_{d-1})=(0,x_1,\ldots,x_{d-1}).$$

The proof requires a solid grasp of the fundamental metric
properties and differential calculus of convex functions; see, e.g.,
\cite{convexcalculus, convexity, metric-geometry}. We summarize the
needed points below for the convenience of the reader.

\subsection{Convex Calculus Basics}
If $S\subseteq \R^{d-1}$ is a convex set and $f:S\rightarrow
\R_{\leq 0}$, then we let $$\epi^* f=\{(y,\mathbf{x})\mid y\in
\R,\,\mathbf x\in S,\,f(\mathbf x)\leq  y\leq 0 \}\subseteq \R^d$$
denote the (truncated) epigraph of $f$ in $\R^{d}$. Following  the
standard convention in the theory of convex analysis, the above
definition of  epigraph is written \emph{upside down}. This is
done, in part, because under this convention, the function $f$ is
convex precisely when $\epi^* f$ is a convex set.

Recall that a function $f:S\rightarrow \R_{\geq 0}$ is
\emph{concave} if and only if $-f$ is convex, which is equivalent to
$$
-\epi^*(-f)=\{(y,\mathbf{x})\mid y\in \R,\,\mathbf x\in S,\,0\leq
y\leq f(\mathbf x) \}\subseteq \R^d
$$
being convex. For $\mathbf z\in \R^{d-1}$, we let $$f'(\mathbf
x;\mathbf z):=\underset{\lambda>0}{\lim_{\lambda\rightarrow
0}}\frac{f(\mathbf x+\lambda \mathbf z)-f(\mathbf x)}{\lambda}$$
denote the {\it onesided directional derivative} of $f$ at $\mathbf
x$ with respect to the direction $\mathbf z$, and then $$-f'(\mathbf
x;-\mathbf z)=\underset{\lambda<0}{\lim_{\lambda\rightarrow
0}}\frac{f(\mathbf x+\lambda \mathbf z)-f(\mathbf x)}{\lambda}.$$
When $d=2$, there are only two directions, and $f_+(x):=f'(x;1)$ is
called the \emph{right derivative} and $f_-(x):=-f'(x;-1)$  the
\emph{left derivative}. It is a basic property of convex functions
that $$\frac{f(\mathbf x+\lambda)-f(\mathbf x)}{\lambda},$$ for
$\lambda >0$, is a non-decreasing function of $\lambda$ (and thus a
non-increasing function of $\lambda>0$ for concave functions $f$),
so that $f'(\mathbf x;\mathbf z)$ always exists (apart from points
on the boundary of $S$ where $f(\mathbf x+\lambda \mathbf z)$ is
undefined for all $\lambda>0$). Moreover, $-f'(\mathbf x;-\mathbf
z)\leq f'(\mathbf x;\mathbf z)$ with equality occurring precisely
when $f$ is differentiable at $\mathbf x$ in the direction $\mathbf
z$, in which case the usual derivative is equal to $-f'(\mathbf
x;-\mathbf z)=f'(\mathbf x;\mathbf z)$.

At a differentiable point $\mathbf x\in \mathsf{int}\; S\subseteq
\R^{d-1}$, where $\mathsf{int}\; S$ denotes the interior of $S$,
there is a unique tangent hyperplane passing through $(f(\mathbf
x),\mathbf x)\in \R\times \R^{d-1}$, which gives rise to the usual
gradient $\nabla f(\mathbf x)\in \R^{d-1}$, whose $i$-th coordinate is
the usual derivative $f'(\mathbf x;\e_i)$. When $f$ is not
differentiable at $\mathbf x$, there is not a unique tangent
hyperplane passing through $(f(\mathbf x),\mathbf x)\in
\R\times\R^{d-1}$. Instead, there are several supporting hyperplanes
passing through $(f(\mathbf x),\mathbf x)$, each one giving rise to
a different subgradient at $\mathbf x$. We let $\partial f(\mathbf
x)$ be the subdifferential of $f$ at $\mathbf x$, which is the set
of all subgradients $\mathbf x^*\in \R^{d-1}$, formally, all
$\mathbf x^*\in \R^{d-1}$ such that the graph of the affine function
$h(\mathbf z)=f(\mathbf x)+\langle \mathbf x^*,\mathbf z-\mathbf
x\rangle$
is a non-vertical supporting hyperplane to $\epi^* f$ at
$(f(\mathbf x),\mathbf x)$, which can be alternatively phrased as
all $\mathbf x^*\in \R^{d-1}$ such that
$$
f(\mathbf z)\geq f(\mathbf x)+\langle \mathbf x^*,\mathbf z-\mathbf
x\rangle\quad \mbox{ for all }\; \mathbf z\in \R^{d-1}.
$$
When
$d=2$, this is simply the set $\partial f(x)=[f_-(x),f_+(x)]$
consisting of all possible slopes of a tangent line passing through
$(f(x),x)$. For instance, if $f(x)=|x|-C$, then
$\partial f(0)=[-1,1]$, $\partial f(x)=\{-1\}$ for $x<0$, and
$\partial f(x)=\{1\}$ for $x>0$.

When $f$ is convex, it is differentiable a.e.~with $f'$ continuous
on the subset of points where it is defined. In fact, $f$ is
Lipschitz continuous in each variable, and thus absolutely
continuous, so that the Fundamental Theorem of Calculus holds. In
particular, if all partial derivatives are zero a.e., then $f$ must
be a constant function. The subdifferential is continuous in the
sense that, given any point $\mathbf x$ in the interior of the
domain of $f$ and any $\epsilon >0$, there exists a $\delta>0$ such
that \be\label{subdiff-cont}\partial f(\mathbf z)\subseteq \partial
f(\mathbf x)+\mathscr B_\epsilon \quad\mbox{ for all } \mathbf z\in
\mathbf x+\mathscr B_\delta,\ee  where $\mathscr B_\rho$ denotes an
open ball of radius $\rho$ (see e.g. \cite[Corollary
24.5.1]{convexcalculus}.) With regards to minimizing a convex
function, we have the rather striking property that a point $\mathbf
x$ is a global minimum for a convex function $f$ if and only if
$\mathbf x$ is a local minimum, which occurs precisely when $\mathbf
0\in \partial f(\mathbf x)$ (see e.g. \cite[Section
27]{convexcalculus}.)

For a subset $A\subseteq \R^d$ and $\lambda\geq 0$, we let
$(A)_\lambda=\bigcup_{\mathbf x\in A}(\mathbf x+\mathscr B_\lambda)$
denote the neighborhood of $A$ consisting of all points strictly
within distance $\lambda$ from a point of $A$. Then the
\emph{Hausdorff distance} between two sets $A,\,B\subseteq \R^d$ is
defined as $$\mathsf d_H(A,B):=\inf\{\lambda\geq 0\mid A\subseteq
(B)_\lambda \und B\subseteq (A)_\lambda\}.$$ When restricted to
closed subsets of $\R^d$, \ $\mathsf d_H(\cdot,\cdot)$ becomes a
metric; in particular, $\mathsf d_H(A,B)=0$, for closed subsets
$A,\,B\subseteq \R^d$, if and only if $A=B$.  Blaschke's Theorem
(see e.g. \cite{metric-geometry, convexity}) asserts that the
Hausdorff metric space is compact when restricted to convex bodies
all contained within some fixed closed ball in $\R^d$. In
particular, if $A_1\subseteq A_2\subseteq\ldots $ is an increasing
sequence of convex bodies all contained within some fixed closed
ball in $\R^d$, then $A_i\rightarrow A$, where $A$ is the closure of
$\bigcup_{i\geq 1} A_i$ and the limit is with respect to the
Hausdorff metric. Additionally, the limit of convex bodies is again
convex.

\subsection{A Sequence of Lemmas}
Our strategy is to first prove Theorem \ref{Dual-thm} when $A$ and
$B$ are the epigraphs of respective concave functions $f:S\subseteq
\R^{d-1}\rightarrow \R_{\geq 0}$ and $g:T\subseteq
\R^{d-1}\rightarrow \R_{\geq 0}$, and then extend to the more
general case. To do this, we break the majority of the proof into a
series of lemmas. Our first lemma below allows us to restrict to the
case when the domains $S$ and $T$ are homothetic. During the course
of the proof, an outline of the proof of Theorem \ref{Bonnesen-cor}
is recreated.

\medskip

\begin{lemma}\label{projection-is-homothetic-lemma} Let $A,\,B\subseteq \R^d$ be convex bodies. If
\be\label{lemma-bound}|A+B|= \left(M^{1/(d-1)}+N^{1/(d-1)}\right)^{d-1}\left(\frac{|A|}{M}+\frac{|B|}{N}\right),\ee where
$M=|\pi(A)|$ and $N=|\pi(B)|$, then $\pi(A)$ and $\pi(B)$ are homothetic.
\end{lemma}

\begin{proof}This is a simple consequence of compression techniques
and the proof of Bonnesen's Theorem as given in \cite{dancs-uhrin}.
We outline the details here.  Recalling that we have assumed
$\mathbf u=\e_0$ and writing a convex body $A$ using the notation of
\eqref{convex-graph-characterization},   we define
$$\C(A):=\{(y,\mathbf x)\in \R\times \R^{d-1}\mid \mathbf x\in
\pi(A),\,0\leq y\leq v_A(\mathbf x)-u_A(\mathbf x)\}.$$ It is easily
derived (see also \cite{oriol-2d-result}) that
\begin{align}\nn&|\C(A)|=|A| \;\;\mbox{ and } \;\; |\C(B)|=|B|,\\\nn
&M:=|\pi(A)|=|\pi(\C(A))|=\sup\{|(\mathbf x+H)\cap \C(A)|\mid \mathbf x\in \R^d\},\\\nn & N:=|\pi(B)|=|\pi(\C(B))|=\sup\{|(\mathbf x+H)\cap \C(B)|\mid \mathbf x\in \R^d\},\\\nn
&|A+B|\geq |\C(A)+\C(B)|,\end{align}
    where $H=\e_0^\bot$ is the orthogonal space to $\e_0$, which is spanned by
$\e_1,\ldots,\e_{d-1}$. For $z\in \R$, let $A(z)=\C(A)\cap
(z\e_0+H)$ and $B(z)=\C(B)\cap (z\e_0+H)$. Then
\ber\nn|\C(A)+\C(B)|&\geq& \int_{-\infty}^{+\infty}
\underset{x+y=z}{\sup}\left\{|A(x)+B(y)|\right\}
dz\\&\geq&\label{tusk1} \int_{-\infty}^{+\infty}
\underset{x+y=z}{\sup}\left\{\left(|A(x)|^{1/(d-1)}+|B(y)|^{1/(d-1)}
\right)^{d-1}\right\} dz \\ &\geq& \label{tusk2}
\left(M^{1/(d-1)}+N^{1/(d-1)}\right)^{d-1}\left(\frac{|A|}{M}+\frac{|B|}{N}\right),
\eer where \eqref{tusk1} follows by the Brunn-Minkowski Theorem
applied to each $A(x)+B(y)$, and \eqref{tusk2} follows by
\cite[Theorem 2.1]{dancs-uhrin} (as in the proof of Bonnesen's Bound
given in \cite{dancs-uhrin}). Consequently, in view of
\eqref{lemma-bound}, we see that equality must hold in
\eqref{tusk1}. The remainder of the proof now follows easily from
the following two basic claims concerning convex bodies.

\subsection*{Claim 1}
If  $X,\,Y\subseteq \R^{d-1}$ are convex bodies that are not
homothetic, then  there exists $\delta>0$ such that no two convex
bodies $C,\,D\subseteq \R^{d-1}$ with $\mathsf d_H(X,C)<\delta$ and
$\mathsf d_H(Y,D)<\delta$ are homothetic.

\begin{proof}
If the claim is false, then there exist two sequences of convex
bodies $\{C_i\}_{i\geq 1}$ and $\{D_i\}_{i\geq 1}$ such that
$C_i\rightarrow X$, $D_i\rightarrow Y$ and, for each $i\geq 1$, \
$C_i$ and $D_i$ are homothetic, so that $D_i=\alpha_iC_i+\mathbf
x_i$ for some $\alpha_i>0$ and $\mathbf x_i\in \R^{d-1}$. Since each
of the sequences $\{C_i\}_{i\geq 1}$ and $\{\alpha_iC_i+\mathbf
x_i\}_{i\geq 1}$ converges to a convex body, it is easily verified
that $\alpha_i\to \alpha$ and $\x_i\to \x$ for some $\alpha >0$ and
$\mathbf x\in \R^{d-1}$. Hence $Y=\alpha X+\mathbf x$, contrary to
the hypothesis.
\end{proof}

\subsection*{Claim 2} For any $\delta>0$, there exists an
$\epsilon>0$ such that $\mathsf d_H(A(0),A(x))<\delta$ and $\mathsf
d_H(B(0),B(y))<\delta$ for all $x,\,y\in [0,\epsilon)$.

\begin{proof} If the claim fails for (say) $A$, then we can find a
sequence $x_1>x_2>\ldots $, where $x_i\in \R_{> 0}$, such that
$x_i\rightarrow 0$ and $\mathsf d_H(A(0),A(x_i))\geq \delta>0$ for all
$i$. Since $x_1>x_2>\ldots $, it follows from the definition of
$A(x)$ that $A(x_1)\subseteq A(x_2)\subseteq \ldots $. Thus
$A(x_i)\rightarrow A'$, where $A'$ is the closure of $\bigcup_{i\geq
1} A(x_i)$. Since $x_i\rightarrow 0$ with $x_i> 0$, it follows
that $\bigcup_{i\geq 1} A(x_i)$ consist of all points $\mathbf x\in
\R^{d-1}$ such that $(y,\mathbf x)\in \C(A)$ for some $y>0$.
Consequently, as $\C(A)=-\epi^*(u_A-v_A)$ is a convex body (both
$-v_A$ and $u_A$ are convex functions), so that $u_A-v_A\leq 0$
cannot be the constant zero function, it follows by a simple
argument that $\mathsf{int}(A(0))\subseteq \bigcup_{i\geq 1}
A(x_i)$, whence $A'=A(0)$. But since $A(x_i)\rightarrow A'=A(0)$, it
now follows that $\mathsf d_H(A(x_i),A(0))\rightarrow 0$,
contradicting that $\mathsf d_H(A(0),A(x_i))\geq \delta>0$ for all
$i$. This completes the claim.
\end{proof}

\bigskip

We now complete the proof the Lemma. If by contradiction
$A(0)=\pi(A)$ and $B(0)=\pi(B)$ are not homothetic, then, by Claims
1 and 2 (take $X=A(0)$ and $Y=B(0)$ in Claim 1 to find the $\delta$
to be used for Claim 2), there is  some $\epsilon>0$ such that
$A(x)$ and $B(y)$ are not homothetic for all $x,\,y\in
[0,\epsilon)$. As a result, the application of the Brunn-Minkowski
Theorem to \eqref{tusk1} yielded a strict inequality for all for
$z\in [0,\epsilon)$, whence equality in \eqref{tusk1} is impossible,
contrary to our assumption.
\end{proof}

\medskip

The following lemma shows that vertical stretching preserves
equality \eqref{dual-hypothesis} provided $\pi(A)$ and $\pi(B)$ are
homothetic, which we will be able to assume using Lemma
\ref{projection-is-homothetic-lemma}. Not only does this show that
the sets described by Theorem \ref{Dual-thm} satisfy the equality
\eqref{dual-hypothesis}, but it will also play an important role in
the other direction of the proof of Theorem \ref{Dual-thm}, allowing
us to consider convex bodies sufficiently stretched and thereby
resolve a delicate technical difficulty with ease.

\medskip

\begin{lemma}\label{lemma-stretching-perseves-bonnbound} Let
$A,\,B,\,A',\,B'\subseteq \R^d$ be convex bodies and suppose that
$A$ and $B$ are vertical stretchings of $A'$ and $B'$, respectively.
Then
\ber &&|A+B|-\left(M^{1/(d-1)}+N^{1/(d-1)}\right)^{d-1}\left(\frac{|A|}{M}+\frac{|B|}{N}\right)\geq\nn\\\nn
&&|A'+B'|-\left(M^{1/(d-1)}+N^{1/(d-1)}\right)^{d-1}\left(\frac{|A'|}{M}+\frac{|B'|}{N}\right)
,\eer
 where
$M=|\pi(A)|=|\pi(A')|$ and $N=|\pi(B)|=|\pi(B')|$, with equality if and only if $\pi(A)=\pi(A')$ and $\pi(B)=\pi(B')$ are homothetic.
\end{lemma}

\begin{proof}
Suppose that $A$ is a stretching of $A'$ of amount $\alpha$ and $B$
is a stretching of $B'$ of amount $\beta$, where $\alpha\geq 0$ and
$\beta\geq 0$. Then
\ber\label{five}|A|&=&|A'|+|\pi(A')|\alpha=|A'|+M\alpha \und \\
|B|&=&\label{four}|B'|+ |\pi(B')|\beta=|B'|+N\beta.\eer
For $\mathbf
z\in \pi(A+B)$, observe that \be\label{stave}\pi^{-1}(\mathbf z)\cap
(A+B)=\bigcup_{\mathbf x+\mathbf y=\mathbf z}
\left((\pi^{-1}(\mathbf x)\cap A)+(\pi^{-1}(\mathbf y)\cap
B)\right).\ee Both $\pi^{-1}(\mathbf x)\cap A$ and $\pi^{-1}(\mathbf
y)\cap B$ are vertical line segments (as $A$ and $B$ are convex).
Moreover, since $A+B$ is also convex, their union in \eqref{stave}
must again be a vertical line segment. The vertical line segment
$\pi^{-1}(\mathbf x)\cap A$ is obtained by extending  the line
segment $\pi^{-1}(\mathbf x)\cap A'$ by an additional length of
$\alpha$ appended onto the top of the segment $\pi^{-1}(\mathbf
x)\cap A'$; the line segment $\pi^{-1}(\mathbf x)\cap B$ is likewise
obtained from  $\pi^{-1}(\mathbf x)\cap B'$ appending on an
additional length of $\beta$ to the top of
 $\pi^{-1}(\mathbf x)\cap B'$. Thus, since the union in
 \eqref{stave} is a single vertical line segment, it follows that
 $$|\bigcup_{\mathbf x+\mathbf y=\mathbf z} \left((\pi^{-1}(\mathbf
 x)\cap A)+(\pi^{-1}(\mathbf y)\cap B)\right)|=|\bigcup_{\mathbf
 x+\mathbf y=\mathbf z} \left((\pi^{-1}(\mathbf x)\cap
 A')+(\pi^{-1}(\mathbf y)\cap B')\right)|+(\alpha+\beta)$$  for each
 $\mathbf z\in \pi(A+B)=\pi(A'+B')=\pi(A')+\pi(B')$. Consequently,
\ber|A+B|&=&\nn|A'+B'|+|\pi(A')+\pi(B')|(\alpha+\beta)\\\label{three}
&\geq &|A'+B'|+(M^{1/(d-1)}+N^{1/(d-1)})^{d-1}(\alpha+\beta),\eer
where \eqref{three} is obtained by applying the Brunn-Minkowski
Theorem, with equality if and only if $\pi(A)=\pi(A')$ and
$\pi(B)=\pi(B')$ are homothetic.

From \eqref{three}, we conclude that
\ber |A+B|-|A'+B'|&=&|\pi(A')+\pi(B')|(\alpha+\beta)\nn\\
&\geq& \label{one} \left(M^{1/(d-1)}+ N^{1/(d-1)}\right)^{d-1}(\alpha+\beta),\eer
with equality if and only if $\pi(A)=\pi(A')$ and $\pi(B)=\pi(B')$ are homothetic.
Also, \eqref{five} and \eqref{four} yield
\ber
&&\nn\left(M^{1/(d-1)}+
N^{1/(d-1)}\right)^{d-1}\left(
\left(\frac{|A|}{M}
+\frac{|B|}{N}\right)-
\left(\frac{|A'|}{M}
+\frac{|B'|}{N}\right)\right)=\\\label{two} && \left(M^{1/(d-1)}+
N^{1/(d-1)}\right)^{d-1}(\alpha+\beta).\eer
Comparing \eqref{one} and \eqref{two} completes the lemma.
\end{proof}

\medskip

Lemma \ref{lemma-basecase-d=2} provides the base case for the
inductive proof of Lemma \ref{lemma-basecase}, which will be our
main argument, combined with standard approximation arguments, for
characterizing the case of equality in Bonnesen's Bound for
epigraphs.

\medskip

\begin{lemma}\label{lemma-basecase-d=2}
Let $m,\,n> 0$ and let $f: [0,m]\rightarrow \R_{\geq 0}$ and $g:[0,n]\rightarrow \R_{\geq 0}$ be concave functions.
Let $A,\,B\subseteq \R^2$ be defined as $A=-\epi^*( -f)$ and $B=-\epi^*(- g)$.
\begin{itemize}
\item[(a)] Then \be\label{syrup}|A+B|\geq (m+n)\left(\frac{|A|}{m}+\frac{|B|}{n}\right)
+\Delta,\ee where $$\Delta= \left(nf(m)-\frac{n}{m}\int_0^m f(x)\;dx\right)+\left(m g(0)-\frac{m}{n}\int_0^n g(x) \;dx\right).$$

\item[(b)]In particular, if $f'_+(x)\geq g'_+(y)+\epsilon$ for all $x\in [0,m)$ and $y\in [0,n)$, where $\epsilon\geq 0$, then
$$|A+B|\geq (m+n)\left(\frac{|A|}{m}+\frac{|B|}{n}\right)
+\frac{mn}{2}\epsilon.$$
\end{itemize}
\end{lemma}

\begin{proof}
We first observe that \ber \nn A+( [0,g(0)] \times\{0\})&\subseteq&
(A+B)\cap (\R\times [0,m]),\\
\nn ( [0,f(m)]\times\{m\} )+B&\subseteq& (A+B)\cap (\R\times [m,m+n] ). \eer
Therefore,
\ber\nn
|A+B|&\geq& |A+([0,g(0)]\times \{0\})|+|([0,f(m)]\times \{m\})+B|
\\\nn &=& (|A|+m g(0))+(|B|+n f(m)),\label{old-estimate}\eer
and we complete the proof of part (a) as follows:
\begin{eqnarray}\nn
|A+B|-(m+n)(\frac{|A|}{m}+\frac{|B|}{n} )
& \ge  &|A|+mg(0)+|B|+n f(m)-(m+n)(\frac{|A|}{m}+\frac{|B|}{n})\\\nn
&=&mg(0)+n f(m)-\frac{n}{m}|A|-\frac{m}{n}|B|\\\nn
&=&
n\left(f(m)-\frac{|A|}{m}\right)+m\left(g(0)-\frac{|B|}{n}\right)=\Delta.
\end{eqnarray}

It remains to prove part (b). Thus suppose $f'_+(x)\geq
g'_+(y)+\epsilon$ for all $x\in [0,m)$ and $y\in [0,n)$, where
$\epsilon\geq 0$. The product of absolutely continuous functions
defined over a closed, bounded  interval is absolutely continuous on this
interval. Thus, since $f:[0,m]\rightarrow \R_{\geq 0}$ is a concave
function, and thus absolutely continuous (and hence differentiable
a.e.), it follows that $x f(x):[0,m]\rightarrow \R_{\geq 0}$ is also
absolutely continuous (and hence differentiable a.e.). As a result,
noting that $(x f(x))'=f(x)+x f'(x)$ a.e., it follows from the
Fundamental Theorem of Calculus that $$f(m)=\frac{1}{m}\int_{0}^m (x
f(x))'\;dx=\frac{1}{m}\int_{0}^m x f'(x)\;dx+\frac{1}{m}\int_{0}^m
f(x)\;dx.$$  Hence we may rewrite $\Delta$ as
$$\Delta=\frac{n}{m}\int_0^m yf'(y)\;dy-\frac{m}{n}\int_0^n g(x)\;dx
+m g(0).$$ Applying the substitution $y\mapsto \frac{m}{n}x$ to the
first integral and using the fact that $f'(x)=f'_+(x)$ a.e., we
obtain
\be\label{stickpep}\Delta=m g(0)+\frac{m}{n}\int_0^n \left(x
f'_+(\frac{m}{n}x)-g(x)\right)\;dx.\ee

Since  $f'_+(x)\geq g'_+(y)+\epsilon$ for all $x\in [0,m)$ and $y\in
[0,n)$, it follows that \be\label{fun0}xf'_+(\frac{m}{n}x)\geq x
g'_+(0)+x\epsilon,\ee for all $x\in [0,n)$.
Since $g$ is concave,
$\frac{g(x)-g(0)}{x}$ is a non-increasing function of $x$, whence
\be\label{fun1}g'_+(0)=\underset{\lambda>0}{\lim_{\lambda\to
0}}\frac{g(\lambda)-g(0)}{\lambda}\geq \frac{g(x)-g(0)}{x}\ee for
all $x\in (0,n)$.
Applying the estimates \eqref{fun0} and \eqref{fun1} to \eqref{stickpep}, we obtain
\ber\nn\Delta&\geq& m g(0)+\frac{m}{n}\int_0^n \left(x
g'_+(0)+x\epsilon-g(x)\right)\;dx\\
&\geq& m g(0)+\frac{m}{n}\int_0^n
\left(-g(0)+x\epsilon\right)\;dx=\frac{mn}{2}\epsilon,
\eer
 which combined with \eqref{syrup} implies the desired bound.
\end{proof}

\medskip

The proof of the following lemma essentially contains a proof of
Theorem \ref{Bonnesen-cor} for $d\geq 3$ using the case $d=2$ as the
base of an inductive argument. The inductive application of Theorem
\ref{Bonnesen-cor} is used to make a kind of  $(d-2)$-dimensional
compression possible.

\medskip

\begin{lemma}\label{lemma-basecase}
Let $d\geq 2$, let $m,\,n> 0$ and let $f: [0,m]^{d-1}\rightarrow \R_{\geq 0}$ and $g:[0,n]^{d-1}\rightarrow \R_{\geq 0}$ be concave functions.
Let $A,\,B\subseteq \R^d$ be defined as $A=-\epi^*(-f)$ and $B=-\epi^*(-g)$. Suppose
 $$f'(\mathbf x;\e_1)\geq g'(\mathbf y;\e_1)+\epsilon
 \quad\mbox{ for all }\;\mathbf x\in [0,m)^{d-1}\und \mathbf y\in [0,n)^{d-1},$$ where $\epsilon\geq 0$.
Then $$|A+B|\geq (m+n)^{d-1}\left(\frac{|A|}{m^{d-1}}+\frac{|B|}{n^{d-1}}\right)
+\frac{mn}{2}(m+n)^{d-2}\epsilon.$$
\end{lemma}

\begin{proof}When $d=2$, Lemma \ref{lemma-basecase-d=2} yields the
desired bound. We assume $d\geq 3$ and proceed by induction on $d$.
For $x\in [0,m]$ and $y\in [0,n]$, let $f_x:[0,m]^{d-2}\rightarrow
\R_{\geq 0}$ and $g_y:[0,n]^{d-2}\rightarrow \R_{\geq 0}$ be defined
by $f_x(x_1,\ldots,x_{d-2})=f(x_1,\ldots,x_{d-2},x)$ and
$g_y(y_1,\ldots,y_{d-2})=g(y_1,\ldots,y_{d-2},y)$. Then
$-\epi^*(-f_x)=(\R^{d-1}\times\{x\})\cap A$ is the $x$-section of
$A$, and we will denote this set by $A(x)=(\R^{d-1}\times\{x\})\cap A$.
Likewise define $B(y)=-\epi^*(-g_y)=(\R^{d-1}\times\{y\})\cap B$
and, for $z\in [0,m+n]$, let $(A+B)(z)=(\R^{d-1}\times\{z\})\cap
(A+B)$. Then $(A+B)(z)=\bigcup_{x+y=z}(A(x)+B(y))$. Consequently,
\be\label{measure-sections}|A+B|\geq
\int_0^{m+n}\underset{x+y=z}\sup\{|A(x)+B(y)|\}\;dz.\ee

By induction hypothesis, we know
$$|A(x)+B(y)|\geq
(m+n)^{d-2}\left(\frac{|A(x)|}{m^{d-2}}+\frac{|B(y)|}{n^{d-2}}\right)
+\frac{mn}{2}(m+n)^{d-3}\epsilon ,$$ for all $x\in [0,m]$ and $y\in
[0,n]$. Combining the above inequality  with
\eqref{measure-sections} gives
\begin{align}\label{measure-sections-II}|A+B|\geq
\int_0^{m+n}\underset{x+y=z}\sup\left\{\frac{(m+n)^{d-2}}{m^{d-2}}|A(x)|+\frac{(m+n)^{d-2}}{m^{d-2}}
|B(y)|\right\}\;dz\\ +\frac{mn}{2}(m+n)^{d-2}\epsilon.\nn\end{align}
Let $\tilde{f}:[0,m]\rightarrow \R_{\geq 0}$ be the function defined
by $$\tilde{f}(x)=\frac{(m+n)^{d-2}}{m^{d-2}}|A(x)|$$ and let
$\tilde{g}:[0,n]\rightarrow \R_{\geq 0}$ be the function defined by
$$\tilde{g}(y)=\frac{(m+n)^{d-2}}{m^{d-2}}
|B(y)|.$$
Let $\tilde{A}=-\epi^*(-\tilde{f})$ and
$\tilde{B}=-\epi^*(-\tilde{g})$. As $A$ and $B$ are convex bodies,
the functions  $|A(x)|$ and $|B(y)|$ are both integrable, and thus
$\tilde{f}$ and $\tilde{g}$ as well. Moreover,
\ber\nn|\tilde{A}+\tilde{B}|&=&\int_{0}^{m+n}
\underset{x+y=z}{\sup}\{\tilde{f}(x)+\tilde{g}(y)\}\;dz\\
&=&
\int_{0}^{m+n} \underset{x+y=z}{\sup}\left\{\frac{(m+n)^{d-2}}{m^{d-2}}|A(x)|+\frac{(m+n)^{d-2}}{m^{d-2}}
|B(y)|\right\}\;dz.\label{stuckkart}
\eer
Applying Theorem \ref{Bonnesen-cor}, which (as mentioned in the
introduction) holds more generally for any compact subsets $A$ and
$B$, to $\tilde{A}+\tilde{B}$, we conclude that
\ber\nn
|\tilde{A}+\tilde{B}|&\geq&
(m+n)\left(\frac{|\tilde{A}|}{m}+\frac{|\tilde{B}|}{n}\right)\\
&=&(m+n)\left(\frac{(m+n)^{d-2}}{m^{d-1}}\int_0^m |A(x)|\;dx+\frac{(m+n)^{d-2}}{n^{d-1}}\int_0^n |B(y)|\;dy\right)\nn\\
&=&(m+n)^{d-1}\left(\frac{|A|}{m^{d-1}}+\frac{|B|}{n^{d-1}}\right).\label{stuckart}
\eer
Combining \eqref{measure-sections-II}, \eqref{stuckkart}  and \eqref{stuckart} yields the desired
lower bound for $|A+B|$, completing the proof.
\end{proof}

\medskip

\subsection*{Completion of the Proof}

We can now proceed with the proof of Theorem \ref{Dual-thm}, first
in the case when  $A$ and $B$ are both epigraphs.

\medskip

\begin{lemma}\label{lemma-function-case}
Let $d\geq 2$, let $S,\,T\subseteq \R^{d-1}$ be convex bodies, and
let $f: S\rightarrow \R_{\geq 0}$ and $g:T\rightarrow \R_{\geq 0}$
be concave functions. Let $A,\,B\subseteq \R^d$ be defined as
$A=-\epi^*(-f)$ and $B=-\epi^*(-g)$. If $$|A+B|=
(|S|^{1/(d-1)}+|T|^{1/(d-1)})^{d-1}\left(\frac{|A|}{|S|}+
\frac{|B|}{|T|}\right),$$ then $S$ and $T$ are homothetic and the
graphs of  $f$ and $g$ are also homothetic, i.e., $$f(\mathbf
x)=\lambda g(\frac{1}{\lambda}(\mathbf x-\mathbf x_0))+C\quad\mbox{
for all }\; \mathbf x\in S,$$ where $C=\frac{|A|}{|S|}-\frac{\lambda
|B|}{|T|}$ and $S=\lambda T+\mathbf x_0$ for some $\lambda>0$ and
$\mathbf x_0\in \R^d$.
\end{lemma}

\begin{proof}
From Theorem \ref{Bonnesen-cor}, we know that
$$|A+B|\geq
(|S|^{1/(d-1)}+|T|^{1/(d-1)})^{d-1}
\left(\frac{|A|}{|S|}+\frac{|B|}{|T|}\right).$$
 We wish to
characterize when equality holds. By Lemma
\ref{projection-is-homothetic-lemma}, equality in the bound implies
$S$ and $T$ are homothetic, say $S=\lambda T+\mathbf x_0$ with
$\lambda>0$ and $\mathbf x_0\in \R^d$. Hence
\be\label{bebo}|S|=\lambda^{d-1}|T|.\ee

By translating appropriately, we may w.l.o.g. assume $\mathbf
x_0=0$, so that $S=\lambda T$. It remains to show that the graphs
of $f$ and $g$ are homothetic, that is, that $f(\mathbf
x)=\lambda g(\frac{1}{\lambda}\mathbf x)+C$ for all $\mathbf x\in
S$, where $C\in \R$ is some constant. To calculate what this
constant must be, we have only to note that $$|A|=\int_S f(\mathbf
x)\;d\mathbf x=\int_S \lambda g(\frac{1}{\lambda}\mathbf
x)+C\;d\mathbf x=|\lambda B|+|S|C=\lambda^d|B|+|S|C,$$ and combine
this with \eqref{bebo}, which gives $C=\frac{|A|}{|S|}-\frac{\lambda
|B|}{|T|}$.

Let $\tilde{g}:S\rightarrow \R_{\geq 0}$ be the concave function
defined by $\tilde{g}(\mathbf x)=\lambda g(\frac{1}{\lambda}\mathbf
x)$. Since $f$ and $g$ are concave functions, they are Lipschitz
continuous in each variable, and thus absolutely continuous.
Furthermore,  $\tilde{g}'(\mathbf
x;\e_j)=g'(\frac{1}{\lambda}\mathbf x;\e_j)$ for all $j\in [1,d-1]$
and all $\mathbf x\in S$. Consequently, if $f'(\mathbf
x;\e_j)=g'(\frac{1}{\lambda}\mathbf x;\e_j)=\tilde{g}'(\mathbf
x;\e_j)$ for all $j\in [1,d-1]$ and a.e. $x\in \mathsf{int}(S)$,
then the Fundamental Theorem of Calculus would imply $f(\mathbf
x)=\lambda g(\frac{1}{\lambda}\mathbf x)+C$ for some constant $C$,
as desired. Therefore, if the statement of the lemma is false, then
there must be some differentiable point $\mathbf x_0\in S$,
contained in the interior of $S$ (as the boundary of $S$ has measure
zero), such that w.l.o.g. $f'(\mathbf x_0;\e_1)\geq
g'(\frac{1}{\lambda}\mathbf x_0;\e_1)+2\epsilon$ with $\epsilon>0$.
In view of \eqref{subdiff-cont}, we can find a small neighborhood
around $\mathbf x_0\in S$ in which $f'(\mathbf x;\e_1)\geq
f'(\mathbf x_0;\e_1)-\frac{\epsilon}{2}$ for $\mathbf x$ in this
neighborhood, as well as a small neighborhood around
$\frac{1}{\lambda}\mathbf x_0\in T$ in which $g'(\frac{1}{\lambda}\mathbf x;\e_1)\leq
g'(\frac{1}{\lambda}\mathbf x_0;\e_1)+\frac{\epsilon}{2}$ for
$\mathbf x$ in this neighborhood. Restricting to smaller
neighborhoods as need be, we can thus find a pair of homothetic
boxes $\mathbf x_0+[-\frac12 \delta,\frac12 \delta]^{d-1}\subseteq
S$ and $\frac{1}{\lambda}\mathbf
x_0+[-\frac{1}{2\lambda}\delta,\frac{1}{2\lambda}\delta]^{d-1}\subseteq
T$ such that \be\label{septre}f'(\mathbf x;\e_1)\geq g'(\mathbf
y;\e_1)+\epsilon\quad\mbox{  for all }\mathbf x\in \mathbf
x_0+[-\frac12\delta,\frac12\delta]^{d-1}\und \mathbf y\in
\frac{1}{\lambda}\mathbf
x_0+[-\frac{1}{2\lambda}\delta,\frac{1}{2\lambda}\delta]^{d-1},\ee
where $\epsilon>0$.

The remainder of the argument is now similar to a standard
inner/outer measure approximation to evaluate a Lebesgue integrable
function; see, e.g. \cite{Rudin}. For $k\in \{0,1,2,\ldots\}$,
partition  $\R^{d-1}$ into a grid using boxes of the form $\mathbf
z+\frac{1}{2^k}[-\frac12\delta,\frac12\delta]^{d-1}$ such that no
two boxes share an interior point and such that $\mathbf
x_0+[-\frac12\delta,\frac12\delta]^{d-1}$ is a union of some subset
of these  boxes. Let $\mathcal B_k$ be the collection of all theses
boxes wholly contained in $S$ and, for each box $\mathfrak b\in
\mathcal B_k$, let $A_{\mathfrak b}\subseteq A$ be the subset
$(\R\times \mathfrak b)\cap A$, which corresponds to the epigraph of
$f$ restricted to the domain $\mathfrak b\subseteq S$. Also, let
$\mathcal B'_k\subseteq \mathcal B_k$ be those boxes whose union is
$\mathbf x_0+[-\frac12\delta,\frac12\delta]^{d-1}$.

Let $\frac{1}{\lambda}\mathcal{B}_k=\{\frac{1}{\lambda}\mathfrak
b\mid \mathfrak b\in \mathcal B_k\}$. Thus $\frac{1}{\lambda}
\mathcal{B}_k$ consists of boxes of the form $\mathbf
z+\frac{1}{2^k}
[-\frac{1}{2\lambda}\delta,\frac{1}{2\lambda}\delta]^{d-1}$, wholly
contained in $T$, such that no two boxes share an interior point and
such that the union of boxes from $\frac{1}{\lambda}\mathcal B'_k$
is equal to $\frac{1}{\lambda}\mathbf x_0+
[-\frac{1}{2\lambda}\delta,\frac{1}{2\lambda}\delta]^{d-1}$. For
each box $\mathfrak b\in \mathcal B_k$, let $B_{\mathfrak
b}\subseteq B$ be the subset $(\R\times \frac{1}{\lambda}\mathfrak
b)\cap B$, which corresponds to the epigraph of $g$ restricted to
the domain $\frac{1}{\lambda}\mathfrak b\subseteq T$. Let
$$m_k=\frac{\delta}{2^{k}}\und n_k=\frac{\delta}{\lambda 2^{k}}$$
be, respectively, the length of each side of the boxes $\mathfrak
b\in \mathcal B_k$ and the length of each side of the boxes
$\frac{1}{\lambda}\mathfrak b\in  \frac{1}{\lambda}\mathcal B_k$.
Thus $$|\mathfrak b|=m_k^{d-1}\und |\frac{1}{\lambda}\mathfrak
b|=n_k^{d-1}\quad\mbox{ for } \mathfrak b\in \mathcal B_k.$$

It is now easily seen that $\bigcup_{\mathfrak b\in \mathcal
B_k}(A_{\mathfrak b}+B_{\mathfrak b})\subseteq A+B$ with the
intersection of any two distinct sumsets $A_{\mathfrak
b}+B_{\mathfrak b}$ being a measure zero subset; of course, we can
also use the more accurate estimate
$$(\bigcup_{\mathfrak b\in
\mathcal B'_k}A_\mathfrak b)+(\bigcup_{\mathfrak b\in \mathcal
B'_k}B_{\mathfrak b})\subseteq A+B$$
in place of $\bigcup_{\mathfrak
b\in \mathcal B'_k}(A_{\mathfrak b}+B_{\mathfrak b})$, and its
intersection with all other $A_{\mathfrak b}+B_{\mathfrak b}$, with
$\mathfrak b\in \mathcal B_k\setminus \mathcal B'_k$, will still be
a measure zero subset.  Thus
$$|A+B|\geq \Summ{\mathfrak b\in
\mathcal B_k\setminus \mathcal B'_k}|A_{\mathfrak b}+B_{\mathfrak
b}|+|(\bigcup_{\mathfrak b\in \mathcal B'_k}A_\mathfrak
b)+(\bigcup_{\mathfrak b\in \mathcal B'_k}B_{\mathfrak b}  ) |.$$

As a result, making use of \eqref{septre} and applying Lemma
\ref{lemma-basecase} to
$$(\bigcup_{\mathfrak b\in \mathcal
B'_k}A_\mathfrak b)+(\bigcup_{\mathfrak b\in \mathcal
B'_k}B_{\mathfrak b})$$

and then using Theorem \ref{Bonnesen-cor} for all other $\mathfrak
b\in \mathcal B_k\setminus B'_k$, we obtain \ber\nn &&|A+B|\geq
\Summ{\mathfrak b\in \mathcal B_k\setminus \mathcal B'_k}(m_k+n_k)^{d-1}(\frac{|A_{\mathfrak b}|}{m_k^{d-1}}+
\frac{|B_{\mathfrak b}|}{n_k^{d-1}})+\\\nn
&&(m_0+n_0)^{d-1}\left(\frac{1}{m_0^{d-1}}
\Summ{\mathfrak b\in \mathcal B'_k}|A_{\mathfrak b}|+\frac{1}{n_0^{d-1}}
\Summ{\mathfrak b\in \mathcal B'_k}|B_{\mathfrak b}|\right)+\frac{(m_0n_0)}{2}(m_0+n_0)^{d-2}\epsilon\\\nn
&=&(\frac{m_k+n_k}{m_k})^{d-1}\Summ{\mathfrak b\in \mathcal
B_k\setminus \mathcal B'_k}|A_{\mathfrak
b}|+(\frac{m_k+n_k}{n_k})^{d-1}\Summ{\mathfrak b\in \mathcal
B_k\setminus \mathcal B'_k}|B_{\mathfrak b}| +\\\nn
&&(\frac{m_0+n_0}{m_0})^{d-1}\Summ{\mathfrak b\in \mathcal
B'_k}|A_{\mathfrak b}|+ (\frac{m_0+n_0}{n_0})^{d-1} \Summ{\mathfrak
b\in \mathcal B'_k}|B_{\mathfrak
b}|+\frac{\delta^d(\lambda+1)^{d-2}}{2\lambda^{d-1}}\epsilon.\\\nn
\eer
In view of the definition of $m_k$ and $n_k$, we have
$\frac{m_k+n_k}{m_k}=1+\frac{1}{\lambda}$ and
$\frac{m_k+n_k}{n_k}=1+\lambda$ for all $k\in \{0,1,2\ldots\}$. Thus
the above calculation implies
\be|A+B|\geq  (1+\frac{1}{\lambda})^{d-1}\Summ{\mathfrak b\in \mathcal B_k}|A_{\mathfrak b}|+
(1+\lambda)^{d-1}\Summ{\mathfrak b\in \mathcal B_k}|B_{\mathfrak b}|+\frac{\delta^d(\lambda+1)^{d-2}}{2\lambda^{d-1}}\epsilon.\label{latest}
\ee

As $k\rightarrow \infty$, we see that $\bigcup_{\mathfrak b\in
\mathcal B_k} \mathfrak b$ approaches $S$. More specifically, since
$S$ is a convex body, the difference between $\lim_{k\rightarrow
\infty} \bigcup_{\mathfrak b\in \mathcal B_k} \mathfrak b$ and $S$
is a measure zero subset. Since $T=\frac{1}{\lambda} S$ is just a
dilation of $S$, we likewise see that the difference between
$\lim_{k\rightarrow \infty} \bigcup_{\mathfrak b\in \mathcal B_k}
\frac{1}{\lambda} \mathfrak b$ and $T$ is also a measure zero
subset. Consequently, $\Summ{\mathfrak b\in \mathcal
B_k}|A_{\mathfrak b}|\rightarrow |A|$ and $\Summ{\mathfrak b\in
\mathcal B_k}|B_{\mathfrak b}|\rightarrow |B|$ as $k\rightarrow
\infty$, whence \eqref{latest}, in view of $\epsilon>0$ and
\eqref{bebo}, shows that
\ber\nn|A+B|&\geq&(1+\frac{1}{\lambda})^{d-1}|A|+
(1+\lambda)^{d-1}|B|+\frac{\delta^d(\lambda+1)^{d-2}}{2\lambda^{d-1}}\epsilon\\\nn
&>& (1+\frac{|T|^{1/(d-1)}}{|S|^{1/(d-1)}})^{d-1}|A|+
(1+\frac{|S|^{1/(d-1)}}{|T|^{1/(d-1)}})^{d-1}|B|\\\nn&=&
(|S|^{1/(d-1)}+|T|^{1/(d-1)})^{d-1}\left(\frac{|A|}{|S|}+
\frac{|B|}{|T|}\right),\eer contrary to hypothesis.
\end{proof}

\bigskip

We conclude the section with the proof of Theorem \ref{Dual-thm} for general convex bodies.

\begin{proof}[Proof of Theorem \ref{Dual-thm}]  As noted at the
beginning of Section \ref{main-sec}, we may w.l.o.g. assume $\pi$ is
the vertical projection map with $\mathbf u=\e_0$. Since a pair of
homothetic convex bodies $A'$ and $B'$ attains equality in the
Brunn-Minkowski inequality, and thus also in
\eqref{dual-hypothesis},  Lemma
\ref{lemma-stretching-perseves-bonnbound} shows that the sets
described by Theorem \ref{Dual-thm} all satisfy equality
\eqref{dual-hypothesis}.

It remains to complete the other direction in Theorem
\ref{Dual-thm}, so assume $A,\,B\subseteq \R^d$ are convex bodies
satisfying \eqref{dual-hypothesis}.
Let $S=\pi(A)$ and $T=\pi(B)$, so that $M=|S|$ and $N=|T|$.  In view
of Lemma \ref{projection-is-homothetic-lemma}, it follows that $S$
and $T$ are homothetic convex bodies, say $S=\lambda T+\mathbf x_0$
with $\lambda>0$ and $\mathbf x_0\in \R^d$, and by translating
appropriately, we may w.l.o.g. assume that $\mathbf x_0=0$. Write
$A$ and $B$ using the notation of
\eqref{convex-graph-characterization}. Note that $v_A(\mathbf x)\geq
u_A(\mathbf x)$  and $v_B(\mathbf y)\geq u_B(\mathbf y)$ for all
$\mathbf x\in S$ and $\mathbf y\in T$. Let $$\alpha=\inf_{\mathbf
x\in S}\{v_A(\mathbf x)-u_A(\mathbf x)\}\geq
0\und\beta=\inf_{\mathbf x\in T}\{v_B(\mathbf x)-u_B(\mathbf
x)\}\geq 0.$$ Since $S$ and $T$ are both compact subsets, these
finite infima  are attained by some $\mathbf v\in S$ and $\mathbf
v'\in T$ (which, of course, may not be the only points for which the minimum is attained).

Let $A'\subseteq \R^d$ be the subset with $\pi(A')=\pi(A)=S$ defined by
 $$A'=\{(y,\mathbf x) \in \R\times \R^{d-1}\mid \mathbf x\in \pi(A),\; u_A(\mathbf x)\le y \le v_A(\mathbf x)-\alpha\}.$$
Then $u_{A'}=u_A$ and $v_{A'}=v_A-\alpha$ (in view of the definition
of $\alpha$) so that $A'$ is the maximal vertical `compression' of
$A$. In particular, $A$ is a vertical stretching of $A'$ of amount
$\alpha$. Likewise, let $B'\subseteq \R^d$ be the subset with
$\pi(B')=\pi(B)=T$ defined by  $$B'=\{(y,\mathbf x) \in \R\times
\R^{d-1}\mid \mathbf x\in \pi(B),\; u_B(\mathbf x)\le y \le
v_B(\mathbf x)-\beta\}.$$ The set $B$ is a vertical stretching of $B'$ of amount $\beta$. In view of Lemma \ref{lemma-stretching-perseves-bonnbound}, we find that the pair $A'$ and $B'$ also satisfies Bonnesen' equality \eqref{dual-hypothesis}.

Since $S=\lambda T$, it is easily observed that, if the  graphs
of $v_A$ and $v_B$ are both homothetic as well as the graphs
of $u_A$ and $u_B$, then we can take $\mathbf
v'=\frac{1}{\lambda}\mathbf v$ and, moreover, $A'$ and $B'$ will
then be homothetic convex bodies as  the graphs of $y=v_A(x)$ and $y=u_A(x)-\alpha$  intersect
over the point $\mathbf v$ while  the graphs of $y=v_{B}(x)$ and
$y=u_{B}(x)-\beta$ intersect over the corresponding point $\mathbf
v'=\frac{1}{\lambda}\mathbf v$, which would complete the proof in view
of the comments of the previous paragraph. We proceed to show this
is the case.

In view of  $S$ and $T$ being homothetic and Lemma
\ref{lemma-stretching-perseves-bonnbound}, we see that, to complete the proof, it suffices
to prove  the pair of graphs  $v_A$ and $v_B$ and the pair of graphs  $u_A$ and
$u_B$ are both homothetic for any pair of vertical stretchings
$\tilde{A}$ and $\tilde{B}$ of  $A$ and $B$. Thus, stretching $A$
and $B$ sufficiently, we may w.l.o.g. assume $$\inf_{\mathbf x\in S}
v_A(\mathbf x)>\sup_{\mathbf x\in S} u_A(\mathbf x)\quad\und\quad
\inf_{\mathbf y\in T} v_B(\mathbf y)>\sup_{\mathbf y\in T}
u_B(\mathbf y).$$
Consequently, translating $A$ and $B$ appropriately, we can assume
that $v_A(\mathbf x)>0$ and $u_A(\mathbf x)<0$ for all $\mathbf x\in
S$, and that $v_B(\mathbf y)>0$ and $u_B(\mathbf y)<0$ for all
$\mathbf y\in T$.

Let
\begin{align}A^+=A\cap (\R_{\geq 0}\times \R^{d-1})&\und
A^-=A\cap (\R_{\leq 0}\times \R^{d-1}),\nn\\\nn B^+=B\cap (\R_{\geq
0}\times \R^{d-1})&\und B^-=B\cap (\R_{\leq 0}\times
\R^{d-1}).\end{align}
Then
\begin{align}A^+=-\epi^*(-v_A)&\und
B^+=-\epi^*(-v_B),\nn\\\nn-A^-=-\epi^*(-(-u_A))&\und
-B^-=-\epi^*(-(-u_B)).\nn\end{align}
Since $A$ and $B$ are convex
bodies, we have $v_A$ and $v_B$ being concave functions and  $u_A$
and $u_B$ convex functions, in which case $-u_A$ and $-u_B$ are
concave functions.

Since $A^++B^+\subseteq \R_{\geq 0}\times \R^{d-1}$ and
$A^-+B^-\subseteq \R_{\leq 0}\times \R^{d-1}$, we see that
$(A^++B^+)\cap (A^-+B^-)$ is a measure zero subset. Thus, applying
Theorem \ref{Bonnesen-cor} to $A^++B^+$ and $A^-+B^-$, it follows
that \ber\nn|A+B|&\geq& |A^++B^+|+|A^-+B^-|
\\\nn
&\geq&(M^{1/(d-1)}+N^{1/(d-1)})^{d-1}(\frac{|A^+|}{M}+\frac{|B^+|}{N})+\\&&\nn
(M^{1/(d-1)}+N^{1/(d-1)})^{d-1}(\frac{|A^-|}{M}+\frac{|B^-|}{N})\\\nn&=&
(M^{1/(d-1)}+N^{1/(d-1)})^{d-1}(\frac{|A|}{M}+\frac{|B|}{N}).
\eer
By hypothesis, equality must hold in the above bound, which is only
possible if equality held in both the estimates for $A^++B^+$ and
for $A^-+B^-$. As result, applying Lemma \ref{lemma-function-case}
to $A^++B^+$ and to $(-A^-)+(-B^-)$ shows that the graphs of $v_A$
and $v_B$ are homothetic as well as the graphs of $u_A$ and
$u_B$, completing the proof.
\end{proof}

\section{Equality in the Hyperplane Slice Bonnesen Bound for $d=2$}

In this section, we prove Theorem \ref{thm-2-dim-nondual}, thus
determining the structure of extremal convex bodies satisfying
Theorem \ref{Bonnesen-thm} in dimension $2$.  To do so, by rotating
appropriately, we can w.l.o.g. assume $H=\e_1^\bot$ is the
$\e_0$-axis. We begin with the following lemma, which does not
necessarily hold for higher dimensions.

\begin{lemma}\label{2-dim-lineup-lemma}
Let  $A\subseteq \R^2$ be a convex body and let  $H=\R e_0$. Suppose
$A$ is translated so that $$|H\cap A|=\sup\{|(\mathbf x+H)\cap
A|\mid \mathbf x\in \R^2\}.$$ Then there exists some linear
transformation $\pi:\R^2\rightarrow \R^2$ with $\pi(\R^2)=H=\R\e_0$
and $\pi(A)=H\cap A$.
\end{lemma}

\begin{proof}
Let $H\cap A=[m,n]\times \{0\}$ with $m<n$. Write $A$ using the
notation of \eqref{convex-graph-characterization} (taking $\mathbf
u=\e_0$) and simplifying the notation for $u_A$ and $v_A$ by
defining $u:=u_A$ and $v:=v_A$. Observe that $u(0)=m$ and $v(0)=n$.
To prove the lemma, we need to find a slope $\lambda$ so that the
line passing through $(m,0)$ with slope $\lambda$ as well as the
line passing through $(n,0)$ with slope $\lambda$ are both
supporting/tangent lines to $A$, as then the linear transformation
$\pi:\R^2\rightarrow H$ having the line of slope $\lambda$ as its
kernel will satisfy the conclusions of the lemma. However, in terms
of subdifferentials, this is equivalent to showing $\partial
u(0)\cap -\partial (-v)(0)$ is nonempty.

Define $\tilde{f}:\pi(A)\rightarrow \R$ by
$\tilde{f}(x)=u(x)-v(x)$. Note $-\tilde{f}(x)= |((0,x)+H)\cap
A|$, so that, by hypothesis, $\min
\tilde{f}=\tilde{f}(0)=m-n<0$. As $A$ is convex, we know $-v$
and $u$ are both convex functions. Hence, since the sum of
convex functions remains convex, we see that $\tilde{f}$ is a
convex function. Since $\tilde{f}(x)$ attains its minimum at
$x=0$, we must have $0\in \partial \tilde{f}(0)$, which means
\be\label{toseep} \tilde{f}'_-(0)\leq 0\leq \tilde{f}'_+(0).\ee
From the definition of the one-sided derivative, it follows that
\be\label{pickup}\tilde{f}'_+={u}'_++{(-v)}'_+\und
\tilde{f}'_-={u}'_-+{(-v)}'_-.\ee

Suppose by contradiction that $\partial u(0)\cap -\partial
(-v)(0)=\emptyset$. Then, since $$\partial u(0)=[{u}'_-(0),
{u}'_+(0)]\und -\partial (-v)(0)=[-(-v)'_+(0), -(-v)'_-(0)],$$ we
see that either  $${u}'_+(0)<-(-v)'_+(0)\quad\mbox{ or }\quad \nn
-(-v)'_-(0)<{u}'_-(0).$$ Consequently, it follows in view of
\eqref{pickup} that either
$$\tilde{f}'_+(0)={u}'_+(0)+{(-v)}'_+(0)<0\quad\mbox{ or }\quad
\tilde{f}'_-(0)={u}'_-(0)+ {(-v)}'_-(0)>0,$$  contradicting
\eqref{toseep}.
\end{proof}

\medskip

We now proceed with the simple derivation of Theorem \ref{thm-2-dim-nondual} from Theorem \ref{Dual-thm}.

\medskip

\begin{proof}[Proof of Theorem \ref{thm-2-dim-nondual}]
If $A$ and $B$ are a pair of sets satisfying the description given
by Theorem \ref{thm-2-dim-nondual}, then the equality
\eqref{dim2-hyp} is the same as the equality
\eqref{dual-hypothesis}, which holds for $A$ and $B$ in view of
Theorem \ref{Dual-thm}.

It remains to complete the other direction of Theorem
\ref{thm-2-dim-nondual}, so assume $A,\,B\subseteq \R^2$ are convex
bodies satisfying \eqref{dim2-hyp}. By rotating appropriately, we
can assume w.l.o.g. that $H=\e_1^\bot=\R\e_0$ is the $\e_0$-axis. We
may also w.l.o.g. assume \be\frac{|B|}{N^2}\leq
\frac{|A|}{M^2}.\label{drucker}\ee In view of Lemma
\ref{2-dim-lineup-lemma}, let $\pi:\R^2\rightarrow \R^2$ be a linear
transformation such that $\pi(\R^2)=H=\R\e_0$ and  $\pi(A)=H\cap A$.
Let $N'=|\pi(B)|$. Note, since $H\cap B\subseteq \pi(B)$ and $N=|
H\cap B|$, that $$N\leq N'.$$ Applying Theorem \ref{Bonnesen-cor}
using $\pi$, we conclude that \be |A+B|\geq
(M+N')(\frac{|A|}{M}+\frac{|B|}{N'})\label{ladyjay}.\ee

Let $h:\R_{>0}^2\rightarrow \R$ be defined by
$h(x,y)=(x+y)(\frac{|A|}{x}+\frac{|B|}{y})$. Then
$$h(M,N)=(M+N)(\frac{|A|}{M}+\frac{|B|}{N})=|A+B|$$ by hypothesis.
Letting $h_x(y)=h(x,y)$, we find that
$h'_x(y)=\frac{-|B|x}{y^2}+\frac{|A|}{x}$, which is non-negative
when \be\label{tial}\frac{|B|}{y^2}\leq \frac{|A|}{x^2},\ee and
positive when $\frac{|B|}{y^2}<\frac{|A|}{x^2}$.

If $(x,y)\in \R^2$  satisfies \eqref{tial} with $x,\,y>0$, then
$(x,y')$ will satisfy \eqref{tial} strictly for all $y'>y$.
Consequently, it follows from the above derivative analysis that
$h(x,y')>h(x,y)$ for such $(x,y)$. In particular, in view of
\eqref{drucker} and $N'\geq N$, we see that $h(M,N')\geq h(M,N)$
with equality possible only if $N'=N$. As a result, since
$|A+B|=h(M,N)$ holds with equality by hypothesis, we conclude from
\eqref{ladyjay} that $N=N'$. Therefore, since $H\cap B\subseteq
\pi(B)$ with $|H\cap B|=N=N'=|\pi(B)|$, we see that $\pi(B)\setminus
(H\cap B)$ is a measure zero subset. Thus, since $B\subseteq \R^2$
is a convex body, so that $\pi(B)$ and $H\cap B$ are both closed
intervals in $\R$, it follows that $\pi(B)=H\cap B$. Hence, since we
also have $\pi(A)=H\cap A$ by the choice of $\pi$, we see that
applying Theorem \ref{Dual-thm} with $\pi$ completes the proof.
\end{proof}

\end{document}